\documentclass[11pt]{amsart}

\usepackage{graphicx}
 \usepackage{url}
\usepackage[left=1.5in, right=1.5in, top=1in, bottom=1in]{geometry}

\usepackage{graphicx}
\usepackage{palatino}
\usepackage[all]{xy}
\usepackage{amssymb}
\usepackage{latexsym}
\usepackage{amscd}
\usepackage{color}
\input amssym.def
\input amssym
\def \a{\alpha  }

\def \ra{\rangle}

\newcommand{\striplet}{\mathcal{SW}(m)}
\newcommand{\WW}{\boldsymbol{ \mathcal{W}}}
\newcommand{\triplet}{\mathcal{W}(p)}

\newcommand{\Z}{\mathbb Z}

\newcommand{\N}{{\mathbb Z}_{\ge 0} }
\newcommand{\C}{\mathbb C}
\newcommand{\Q}{\mathcal  Q}

\usepackage{amssymb}
\usepackage{latexsym}
\usepackage{epsfig}
\usepackage{graphics}
\usepackage[all]{xy}
\usepackage{color}
\input amssym.def
\input amssym
\newtheorem{theorem}{Theorem}[section]
\newtheorem{corollary}{Corollary}[section]
\newtheorem{lemma}{Lemma}[section]

\newtheorem{conjecture}{Conjecture}[section]

\newtheorem{observation--theorem}{Observation-Theorem}[section]

\newtheorem{remark}{Remark}[section]
\newtheorem{definition}{Definition}[section]
\newtheorem{proposition}{Proposition}[section]
\def\:{\mbox{\tiny ${\bullet\atop\bullet}$}}
\newtheorem{example}{Example}[section]

\newcommand{\bea}{\begin{eqnarray}}

\newcommand{\eea}{\end{eqnarray}}

\newcommand{\be}{\begin {equation}}

\newcommand{\ee}{\end{equation}}

\newcommand{\W}{\mathcal{W}}

\begin{document}

\title[$C_2$-cofinite $\mathcal{W}$-algebras and their logarithmic representations]
{$C_2$-cofinite $\mathcal{W}$-algebras and their logarithmic representations}

\author{Dra\v zen  Adamovi\'c and Antun  Milas}

\address{Department of Mathematics, University of Zagreb, Croatia}
\email{adamovic@math.hr}

\address{Department of Mathematics and Statistics,
University at Albany (SUNY), Albany, NY 12222}
\email{amilas@math.albany.edu}

 \markboth{Dra\v zen Adamovi\' c and Antun Milas} { }
\bibliographystyle{amsalpha}

\begin{abstract}
We discuss our recent results on the representation theory of $\mathcal{W}$--algebras relevant to Logarithmic Conformal Field Theory. 
First we explain some general constructions of $\mathcal{W}$-algebras coming from screening operators.
Then we review the results on $C_2$--cofiniteness, the structure of Zhu's algebras, and the existence of logarithmic modules for triplet vertex algebras. We propose some conjectures and open problems which put the theory  of triplet vertex algebras into a broader context.  New realizations of logarithmic modules for $\mathcal{W}$-algebras defined via screenings are also presented.
\end{abstract}

\maketitle

\section{Introduction: Irrational $C_2$-cofinite vertex algebras}

Vertex algebras are in many ways analogous to associative algebras, at least from the point of view of representation theory.
Rational vertex operator algebras \cite{Zh}, \cite{ABD} and regular vertex algebras
 have semisimple categories of modules and should be compared to (finite-dimensional) semisimple associative algebras.  If we seek the same analogy with  finite-dimensional non-semisimple associative algebras, we would eventually discover irrational $C_2$-cofinite vertex algebras (the $C_2$-condition guarantees that the vertex algebra has finitely many inequivalent irreducibles \cite{Zh}).
But oddly as it might seem, examples of such vertex algebras are rare and actually not much is known about them.
For instance, it is not even known if there exists an irrational vertex algebra with finitely many indecomposable modules.

Motivated by important works of physicists \cite{FHST}, \cite{FGST1}-\cite{FGST3},  in our recent papers \cite{AM1}, \cite{AM2}, \cite{AM6}, \cite{AM8} (see also \cite{Abe}, \cite{CF}) , among many other things, we constructed new families of irrational $C_2$-cofinite (i.e.,  {\em quasi-rational})  vertex algebras and superalgebras.
The most surprising fact about  quasi-rational vertex algebras is that all known examples are expected to be related to certain finite-dimensional quantum groups (Hopf algebras) via the conjectural Kazhdan-Lusztig correspondences \cite{FGST2}, \cite{FGST3}, \cite{FT} . It is also known that the module category of a $C_2$-cofinite vertex algebra has a natural finite tensor category structure \cite{HLZ}, \cite{Hu1}, although not necessarily rigid \cite{Miy2} (see also \cite{Fu}, \cite{FS} for related categorical issues). 

This note is based on lectures given by the authors at the conference on "Tensor Categories and Conformal Field Theory", June 2011, Beijing.
Thus, except for Section 5 and  some constructions in Sections 2, all the material is based on earlier works by the authors (we should say that some constructions in Section 2 were independently introduced in \cite{FT}). We are indebted to the organizers for invitation to this wonderful conference. We also thank A. Semikhatov,  A. Ga\u\i nutdinov,  A. Tsuchiya,  I. Runkel, Y. Arike, J. Lepowsky, L. Kong and Y-Z. Huang for the interesting discussion during and after the conference. We are also grateful to Jinwei Yang for helping us around in Beijing.

\section{Preliminaries}

This paper deals mainly with the representation theory of certain vertex algebras. Because  we are interested in their
$\mathbb{Z}_{\geq 0}$-graded modules, the starting point is to recall the definition of Zhu's algebra for vertex operator
(super)algebras following \cite{Zh}, \cite{KW}.

Let $(V,Y, {\bf 1}, \omega)$ be a vertex operator algebra. We
shall always assume that $V$ is of CFT type, meaning that it has $\mathbb{N}_{\geq 0}$ grading with the vacuum vector lying on the top component. Let $V =\coprod_{ n \in {\N} } V(n)$.
For $a \in V(n)$, we shall write  ${\rm deg}(a)=n$.
As usual, vertex operator associated to $a \in V$ is denoted by $Y(a,x)$, with the mode expansion
$$Y(a,x)=\sum_{n \in \mathbb{Z}} a_n x^{-n-1}.$$

We define two bilinear maps: $* : V  \times V \rightarrow V$,
$\circ : V \times V \rightarrow V$ as follows. For homogeneous $a,
b \in V$ let
\bea
a* b &&=
 \  \mbox{Res}_x Y(a,x) \frac{(1+x) ^{\deg (a)}}{x}b   \nonumber \\
a\circ b &&=
 \  \mbox{Res}_x Y(a,x) \frac{(1+x) ^{\deg (a)} }{x^2}b  \nonumber
  \eea

Next, we extend $*$ and $\circ$ on $V \otimes V$ linearly, and
denote by $O(V)\subset V$ the linear span of elements of the form
$a \circ b$, and by $A(V)$ the quotient space $V / O(V)$. The
space $A(V)$ has an associative algebra structure (with identity), with the multiplication induced by $*$. Algebra $A(V)$ is called  the
Zhu's algebra of $V$. The image of $v \in V$, under the natural
map $V \mapsto A(V)$ will be denoted by $[v]$.

For a homogeneous $a \in V$ we define $o(a) = a_{\deg(a)-1}$.
In the case when $V ^{\bar 0} = V$, $V$ is a vertex operator algebra and we get the usual definition of Zhu's algebra for vertex operator algebras.

According to \cite{Zh}, there is an one-to-one correspondence between
irreducible $A(V)$--modules and irreducible
$\mathbb{Z}_{\geq 0}$--graded $V$--modules.

Moreover, if $U$ is any $A(V)$--module. There is $ \N$--graded $V$--module $L(U)$ such that the top component $L(U) (0) \cong U$.
$V$ is called rational, if every $\ \N$--graded module is completely reducible.

With $V$ as above, we let $C_2(V)=\langle a_{-2} b : a,b \in V \rangle$, and $\mathcal{P}(V)=V/C_2(V)$.
The quotient space $V /C_2 (V )$ has an algebraic structure of a commutative Poisson
algebra \cite{Zh}. Explicitly, if we denote by $\bar{a}$ the image of $a$ under the natural map $V \mapsto \mathcal{P}(V)$ the poisson bracket is given by
$\{\overline{a}, \overline{b} \} = \overline{ a_0 b}$ and commutative product $\overline{a} \cdot  \overline{b} = \overline{a_{-1} b}$.   From the given definitions it is not hard to construct an increasing filtration of $A(V)$ such that ${\rm gr}A(V)$ maps
onto $A(V)$.

\section{Quantum $\W$-algebras from integral lattices}

\label{part-1}
$\mathcal{W}$-algebras are some of the most exciting objects in representation theory and have been extensively studied from many different point of views. There are several different types of $\mathcal{W}$-algebras in the literature, so to avoid any confusion we stress that  (a) {\em finite} $W$-algebras are certain associative algebras associated to a complex semisimple Lie algebra $\frak{g}$ and a nilpotent element $e \in \frak{g}$ \cite{BT}, \cite{W2}, and can be viewed as deformations of Slodowy's slice, and
(b)  {\em affine} $W$-algebras  are vertex algebras
\footnote{{\em n.b.} For brevity, we shall often use "algebra" and "vertex algebra" when we mean "superalgebra" and  "vertex superalgebra", respectively. From the context it should be clear whether the adjective "super" is needed.}
obtained by Drinfeld-Sokolov reduction from affine vertex algebras \cite{FrB}.  The two algebras
are related via a fundamental construction of Zhu (cf. \cite{Ar} \cite{DeK}).
In this paper, (quantum) $\mathcal{W}$-algebras are vertex algebra generalizations of the affine $W$-algebras. More precisely

\begin{definition}  A $\mathcal{W}$-algebra $V$ is a vertex algebra strongly generated by a finite set
of primary vectors $u^1,...,u^k$. Here strongly generated means that elements of the form
\begin{equation} \label{span}
u^{i_1}_{-j_1} \cdots u^{i_m}_{-j_m}{\bf 1}, \ \ j_1,...,j_m \geq 1
\end{equation}
form a spanning set of $V$. If ${\rm deg}(u_i)=r_i$ we say that $V$ is of type $(2,r_1,...,r_k)$.
\end{definition}

Let us first outline the well-known construction of lattice vertex algebra $V_L$ associated to  a positive definite even lattice
$(L,  \langle  \  \  \rangle)$. We denote by $\mathbb{C}[L]$ the group algebra of $L$.
As a vector space
$$V_L=M(1)  \otimes \mathbb{C}[L], \ \ M(1)=S(\hat{\frak{h}}_{<0}),$$
where $S(\hat{\frak{h}}_{<0})$ is the usual Fock space.
The vertex algebra $V_L$ is known to be rational \cite{D}.
Denote by $L^\circ$ the dual lattice of $L$.  For $\beta \in L^{\circ}$ we have "bosonic"
vertex operators
$$Y(e^{\beta},x)=\sum_{n \in \mathbb{Z}} e^{\beta}_{n} x^{-n-1},$$
introduced in \cite{FLM}, \cite{DL}.  It is also known \cite{D} that all irreducible $V_L$-modules are
given by $V_{\gamma}$, $\gamma \in L^\circ/L$.

Now, we specialize $L=\sqrt{p} Q$,  where $p \geq 2$  and $Q$  is a root lattice (of ADE type).
We should say that this restriction is not that crucial right now, and in fact we
can obtain interesting objects even if the lattice $L$ is (say) hyperbolic.
We equip $V_L$ with a vertex algebra structure  \cite{Bo}, \cite{FLM} (by choosing an appropriate 2-cocycle).
Let $\alpha_i$ denote the simple roots of $Q$. For the conformal vector we conveniently choose
$$\omega=\omega_{st}+\frac{p-1}{2\sqrt{p}} \sum_{\alpha \in \Delta_+} \alpha(-2){\bf 1},$$
where $\omega_{st}$ is the standard  (quadratic) Virasoro generator \cite{FLM}, \cite{LL} . Then $V_L$ is a conformal vertex
algebra of central charge \footnote{Without the linear term the central charge would be $rank(L)$.}  $${\rm rank}(L)+12 (\rho,\rho)(2-p-\frac{1}{p}).$$
Consider the operators
\begin{equation} \label{screening}
e^{\sqrt{p} \alpha_i}_0, \ \  \  \ e^{-\alpha_j/\sqrt{p}}_0, \ \  1 \leq i,j \leq rank(L)
\end{equation}
acting between $V_L$ and $V_L$-modules.
These are the so-called {\em screening operators}. More precisely,

\begin{lemma} For every $i$ and $j$ the operators $e^{\sqrt{p} \alpha_i}_0$  and  $e^{-\alpha_j/\sqrt{p}}_0$ commute with each other, and they both commute with the Virasoro algebra.
\end{lemma}
We shall refer to  $e^{\sqrt{p} \alpha_i}_0$ and $e^{-\alpha_j/\sqrt{p}}_0$, as the {\em long} and {\em short}  screening, respectively.
It is well-known that  the intersection of the kernels of residues of vertex operators is a vertex subalgebra (cf. \cite{FrB}), so the next problem seems very natural to ask

\vskip 2mm

\noindent {\bf Problem 1.} { What kind of vertex algebras can we construct from the kernels of screenings in (\ref{screening})? What choices of  (\ref{screening})  give rise to
$C_2$-cofinite vertex (sub)algebras?}

\vskip 2mm

\subsection{Affine $\mathcal{W}$-algebras}

The above construction with screening operators naturally leads to affine $\mathcal{W}$-algebras.
The affine $\mathcal{W}$-algebra associated to $\hat{\frak{g}}$ at level $k \neq -h^\vee$, denoted by $\mathcal{W}_k({\frak{g}})$
is defined as
$$H^{*}_k(\frak{g}),$$
where the cohomology is taken with respect to a quantized BRST complex  for the Drinfeld-Sokolov hamiltonian reduction \cite{FKW}.
As shown by Feigin and Frenkel (cf. \cite{FKW} and \cite{FrB} and citations therein) this cohomology is nontrivial only in  degree zero. Moreover, it is known that
$\mathcal{W}_k(\frak{g})$ is a  quantum $\mathcal{W}$-algebra (according to our definition) { freely} generated by
${\rm rank}(\goth{g})$ primary fields. Although not evident from our discussion, the vacuum vertex algebra $V_k(\hat{\frak{g}})$ coming from the
affine Kac-Moody Lie algebra $\widehat{\goth{g}}$ enters in the definition of $H^0_k({\frak{g}})$ (see again \cite{FrB}).
It is possible to replace $V_k(\hat{\frak{g}})$ with its irreducible quotient $L_k(\hat{\frak{g}})$, but then the theory becomes much more complicated \cite{Ar}.

An important theorem of B. Feigin and E. Frenkel  \cite{FrB} says that if $k$ is {\bf generic} and $\goth{g}$ is simply-laced, then there is an alternative description of $\mathcal{W}_k(\frak{g})$. For this purpose, we let $\nu={k+h^\vee}$, where $k$ is generic. Then there are appropriately defined screenings
$$e^{-\alpha_i/\sqrt{\nu}}_0 : M(1) \longrightarrow M(1,-\alpha_i/\sqrt{\nu}),$$
such that
$$\mathcal{W}_k(\frak{g})=\bigcap_{i=1}^l {\rm Ker}_{M(1)} (e^{-\alpha_i/\sqrt{\nu}}_0),$$
where $l={\rm rank}(L)$.
If we assume in addition that $\frak{g}$ is simply laced (ADE type) then we also have the following
important duality \cite{FrB}
$$\mathcal{W}_k(\frak{g})=\bigcap_{i=1}^l {\rm Ker}_{M(1)} (e^{\sqrt{\nu} \alpha_i}_0).$$

Now, let us consider the case when $L=\sqrt{p}Q$, $p \in \mathbb{N}$, in connection with the problem we just raised.
Having in mind the previous construction,  it is natural to ask whether $p=k+h^\vee$ is also generic.
For instance,  it is known (cf. \cite{FKRW}) that  $p=1$  is generic. Next result seems to be known in the physics literature

\begin{theorem} \label{non-generic} Let $\frak{g}$ be simply laced. Then $p=k+h^\vee \in \mathbb{N}_{\geq 2}$ is non-generic. More precisely,
$$\bigcap_{i=1}^l {\rm Ker}_{M(1)} e^{-\alpha_i/\sqrt{p}}_0$$
is a vertex algebra containing $\mathcal{W}_k(\frak{g})$ as a proper subalgebra.
\end{theorem}

Interestingly enough, for long screenings, we do expect "genericness"  to hold:
\begin{conjecture} \label{generic} For $p \geq 2$ as above
$$\mathcal{W}_k(\frak{g})=\bigcap_{i=1}^l {\rm Ker}_{M(1)} e^{\sqrt{p} \alpha_i }_0$$
\end{conjecture}

This conjecture is known to be true in the rank one case, where  the kernel of the  long screening is precisely the Virasoro vertex algebra $L(c_{p,1},0)$ of central charge $1-\frac{6(p-1)^2}{p}$. But for the short screening we
obtain the so-called singlet algebra $\overline{M(1)}$ of type $\mathcal{W}(2,2p-1)$, an extension of $L(c_{p,1},0)$  \cite{AM1}, \cite{Ad} (the special case $p=2$ has been extensively studied in \cite{Abe}, \cite{CR}, \cite{W1}, etc.). Both vertex algebras are neither rational nor $C_2$-cofinite.

Let $$h_{r,s} = \frac{ (s p - r) ^2 - (p-1) ^2 }{4 p }. $$

\begin{theorem}\cite{Ad}
Zhu's associative algebra $A(\overline{M(1)})$ is isomorphic to the
commutative algebra   ${\C}[ x, y] / \langle  P(x,y) \rangle$,
where  $\langle  P(x,y) \rangle $ is the principal ideal generated by
\bea \label{ass-poly} P(x,y) = y ^{2} - C_p ( x - h_{p,1}) \prod_{i= 1}
^{p-1} (x- h_{i,1})  ^{2}  \nonumber \qquad ( C_p \ne 0).\eea
\end{theorem}

Now, by using results in Section 2,  we see that irreducible $\overline{M(1)}$--modules are parameterized by zeros of a certain  rational curve in $\C ^2$. We expect that irreducible modules for vertex operator algebras from Theorem \ref{non-generic} also have interesting interpretation in the context of algebraic curves.

\subsection{Further extended affine $\mathcal{W}$-algebras}

Instead of focusing on the charge zero subspace $M(1)$ (the Fock space), nothing prevented us
from considering intersections of the kernels of screenings on the whole lattice vertex algebra $V_L$. Let us first examine the long screenings in this situation. Conjecturally,  we expect  to produce a certain vertex algebra  denoted by  $\mathcal{W}^{\diamond}(p)_Q:=\displaystyle{\bigcap_{i=1}^l {\rm Ker}_{V_L} e^{\sqrt{p} \alpha_i }_0}$, with a large  ideal $I$ such that $W^{\diamond}(p)_Q/I \cong \mathcal{W}_k(\frak{g})$ (the structure of $\mathcal{W}^{\diamond}(p)$ was analyzed in \cite{MP2} in connection to Feigin-Stoyanovsky's principal subspaces \cite{FSE}).

\begin{example} For $Q=A_1$. we have
$$W^{\diamond}(p)_Q=\langle e^{ \sqrt{p} \alpha_1}, \omega \rangle,$$
the smallest conformal vertex subalgebra of $V_L$ containing the vector $e^{\alpha_1}$.
Here $\langle e^{ \sqrt{p} \alpha_1} \rangle$ is the well-known FS principal subspace \cite{FSE} (see also \cite{MP2}).
 \end{example}

\subsection{Maximally extended $\mathcal{W}$-algebras: a conjecture}
\label{extended}
Due to differences already observed in Theorem \ref{non-generic} and Conjecture \ref{generic}, it is not surprising that the conformal vertex algebra
\begin{equation} \label{short}
\mathcal{W}(p)_Q:=\bigcap_{i=1}^l {\rm Ker}_{V_L} e^{-\alpha_j/\sqrt{p}}_0
\end{equation}
will exhibit  properties different to those observed for $\mathcal{W}^{\diamond}(p)_Q$.

We believe the following rather strong conjecture motivated by \cite{FT} holds.
\begin{conjecture}  We have
\begin{itemize}
 \item[(1)] The vertex algebra $\mathcal{W}(p)_Q$ is irrational and $C_2$-cofinite,
 \item[(2)] It is strongly generated by the generators of $\mathcal{W}_k(\frak{g})$ and
finitely many primary vectors,
 \item[(3)] $Soc_{\mathcal{W}_k(\frak{g})} (V_L)=\mathcal{W}(p)_Q$
 \item[(4)] $\mathcal{W}(p)_Q$ admits logarithmic modules of $L(0)$-nilpotent rank at most ${\rm rank}(L)+1$ (for the explanation see Section \ref{log-section}).
 \item[(5)] $\dim A({\triplet}_Q) =\dim \mathcal{P} ({\triplet}_ Q)$.
\end{itemize}
\end{conjecture}

Let us briefly comment on (5) first.  For $V$ a $C_2$--cofinite vertex operator algebra. M. Gaberdiel and T. Gannon in \cite{GG} initiated a relationship between $A(V)$ and $\mathcal{P}(V)$. They raised an interesting question: When does $\dim A(V) = \dim \mathcal{P}(V) $?  For a large family of rational vertex operator algebras of affine type, the equality of dimensions holds (cf. \cite{FL}, \cite{FFL}).  In \cite{AM7}, we studied  this question for $C_2$--cofinite, irrational vertex operator superalgebras, and proved that (5) holds in the rank
one case.

The first half of part (1) of the conjecture is known to be true in general, and this
follows also from (4). We have already shown in \cite{AM1}  the conjecture to
be true for $Q=A_1$.  In this case we write  $\triplet= \triplet_{Q}$ for brevity.

\subsection{Triplet vertex algebra $\triplet$ }

The next result was proven in \cite{AM1} and \cite{AM7}.

 \begin{theorem} The following holds:
 \begin{itemize}
  \item[(1)] $\mathcal{W}(p)$ is  $C_2$--cofinite and irrational.
  \item[(2)] $\triplet$ is strongly generated by $\omega$ and three primary vectors
  $E$, $F$ and $H$ of conformal weight $2 p-1$.
  \item[(3)] $\triplet$ has exactly $2p$ irreducible modules, usually denoted by
   $$ \Lambda(1), \dots, \Lambda(p); \Pi(1), \dots, \Pi(p). $$
  \item[(4)]
$$\dim A(\mathcal{W}(p)) =  \dim \mathcal{P} (\triplet) =  6p-1.$$
\end{itemize}
\end{theorem}

\noindent Let us here recall description of  $C_2$--algebra  $\mathcal{P} (\triplet)$.
Generators of $\mathcal{P}(\mathcal{W}(p))$ are given by
$$\bar{\omega}, \bar{H}, \bar{E}, \bar{F},$$
and the relations are
\bea \label{c2-1}
&&   \bar{E} ^2 = \bar{F} ^2 = \bar{H} \bar{F}= \bar{H} \bar{E}=  0, \nonumber \\
&& \bar{H} ^2 =- \bar{E} \bar{F} = {\nu}  \bar{\omega} ^{ 2p -1} \quad (\nu  \ne 0), \nonumber \\
&&  \bar{\omega} ^ p \bar{H}=\bar{\omega} ^ p \bar{E}= \bar{\omega} ^p\bar{F}=0. \nonumber
\eea

The complete description of the structure of Zhu's algebra was obtained in  \cite{AM7}, where we developed a new method for the determination of Zhu's algebra which was based on a construction of homomorphism of $\Phi : A(\overline{M(1)}) \rightarrow A(\triplet)$. Then we described the kernel of such homomorphism, and by using knowledge of Zhu's algebra $A(\overline{M(1)})$ for the singlet vertex algebra $\overline{M(1)}$ mentioned earlier, we get the following result:
\begin{theorem}\cite{AM1}, \cite{AM7}
Zhu's algebra $A(\mathcal{W}(p))$ decomposes as a direct sum:
$$A(\mathcal{W}(p)) = \bigoplus_{i=2p } ^{3 p-1}   \mathbb{M}_{h_{i,1}}  \oplus \bigoplus _{i=1} ^{p-1}\mathbb{I}_{h_{i,1}} \oplus {\C}_{h_{p,1}},$$
where
$\mathbb{M}_{h_{i,1}}$ ideal isomorphic to matrix algebra $M_2 (\C)$,
 $\mathbb{I}_{h_{i,1}}$ is $2$--dimensional ideal,
${\C}_{h_{p,1}}$ is $1$-dimensional ideal. The structure of Zhu's algebra $A(\triplet)$ implies the existence of logarithmic modules.
\end{theorem}

The similar result was obtained in \cite{AM7} for what we called the super-triplet vertex algebra $\striplet$. The advantage of the method used in \cite{AM7} is that  for the description of Zhu's algebra we don't use any  result about the  existence of logarithmic representations. In our approach,  the existence of logarithmic representations is a consequence of the description of Zhu's algebra.

\begin{corollary} \cite{AM7}.
For every $1 \le i \le p-1$, there exits a logarithmic, self-dual,  $\N$--graded $\triplet$--module denoted by $\mathcal{P}_{i} ^+$ such that the top component $\mathcal{P}_{i} ^+ (0)$ is two-dimensional and
$L(0)$ acts on it (in some basis) as
$$\left(
    \begin{array}{cc}
      h_{i,1} & 1 \\
      0 & h_{i,1} \\
    \end{array}
  \right) .$$
\end{corollary}

\begin{remark}
The vertex algebra $\triplet$ also has $(p-1)$--logarithmic modules $\mathcal{P}_i ^-$ which can not be detected by $A(\triplet)$. These modules can be constructed explicitly as in \cite{AM4} and \cite{NT}. On the other hand, one can apply the Huang-Lepowsky tensor product  $\hat{\otimes}$  \cite{HLZ} and get:
$$ \mathcal{P}_i ^-:= \mathcal{P}^+ _{p-i} \hat{\otimes} \Pi(1). $$
\end{remark}

\begin{remark}
Almost everything in this section can be modified, along the lines of \cite{AM2} \cite{AM3}, to $N=1$ vertex superalgebras, by consideration of odd lattices and by tensoring $V_L$ with the free fermion vertex superalgebra.
\end{remark}

 \section{$\mathcal{W}$-algebra extensions of minimal models}

If we consider $L=\sqrt{pp'} Q$, where $p$ and $p'$ are relatively prime and strictly bigger than one, there are additional degrees of freedom entering the construction of screening operators.
These values allow us to construct more complicated vertex algebras, closely related
to affine $\mathcal{W}$-minimal models.

%It is possible to introduce a version of $\mathcal{W}(p,p')_Q$ at any minimal central charge.
For simplicity we only consider the case $Q=A_1$,  well studied in the physics literature.

 The setup is $L=\sqrt{pp'} \mathbb{Z}\a_1$, $\langle \a_1, \a_1 \rangle=2$. To avoid (annoying) radicals, let $\a = \sqrt{ p p'} \a_1$. Then
 $$ L= \Z \a, \quad \langle \a , \a \rangle = 2 p p'. $$

We construct $V_L$ as before but now we choose $$\omega_{p,p'}=\omega_{st}+\frac{p-p'}{2 pp'} \alpha(-2),$$
such that the central charge is $1-6\frac{(p-p')^2}{pp'}$ (minimal central charges \cite{W1}).
There are again two screening operators here  \cite{FGST2} \cite{FGST3} (cf.  \cite{AM6}, \cite{AM8}):
$${\Q}=e^{\alpha / p'}_0 \ \  {\rm and} \ \  \widetilde{\Q}=e^{-\alpha / p }_0$$
Although the rank is one, the replacement for $\mathcal{W}(p)_Q$ involves
{\em both} screenings, namely
$$\mathcal{W}_{p,p'}:={\rm Ker}_{V_L} {\Q} \cap {\rm Ker}_{V_L} {\widetilde{\Q}}.$$
Compared to $\mathcal{W}(p)_Q$ this vertex algebra is more complicated and it is no longer simple \cite{AM6}.
The inner structure of $V_L$, and of $\mathcal{W}_{p,p'}$, as a Virasoro algebra
module, can be visualized via the following diagram describing the semisimple filtration of $V_L$. Here all $\bullet$ symbols denote highest weight vectors for the Virasoro algebra and they
generate the socle part of $V_L$. Similarly, all $\bigtriangleup$ symbols are representatives of the top part in the filtration, etc.
 $$
 
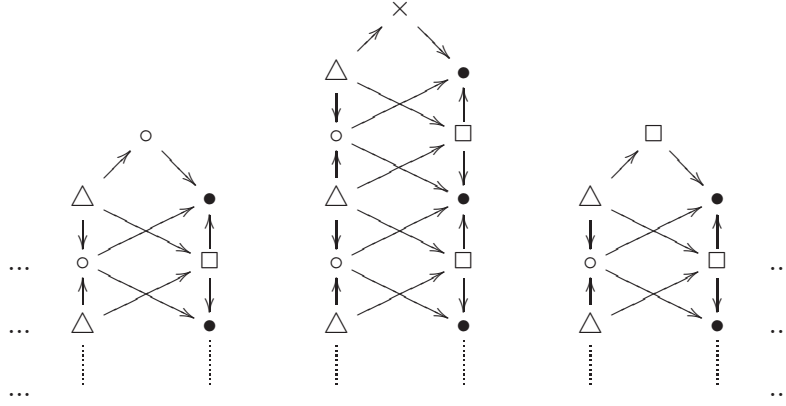
\xymatrix@!0{     & &  &  &  &  & \times  \ar[dr] &  &  & & & & \\
 & &  &  &  &      \bigtriangleup  \ar[ur]  \ar[rrd]  \ar[d] & &  \bullet & & & & & \\
 & & \circ  \ar[dr] &  &   & \circ \ar[rru] \ar[rrd] &  & \Box \ar[u] \ar[d] &   &
&  \Box  \ar[dr] &  &
\\
& \bigtriangleup  \ar[ur]  \ar[rrd]  \ar[d]   &  &  \bullet   &
& \bigtriangleup \ar[urr] \ar[drr] \ar[d]  \ar[u] &  & \bullet   &
& \bigtriangleup  \ar[ur]  \ar[rrd]  \ar[d] & &  \bullet &
\\
... & \circ \ar[rru] \ar[rrd] &  & \Box \ar[u] \ar[d] &
&  \circ \ar[rru] \ar[drr] &  & \ar[u] \Box \ar[d] &
& \circ \ar[rru] \ar[rrd] &  & \Box \ar[u] \ar[d] & ...
\\
... & \bigtriangleup \ar[urr]  \ar[u]  \ar@{.}[d] &  & \bullet  \ar@{.}[d] &
 & \bigtriangleup \ar[urr]  \ar[u]  \ar@{.}[d] &  & \bullet \ar@{.}[d]  &
& \bigtriangleup \ar[urr]  \ar[u]  \ar@{.}[d] &  & \bullet  \ar@{.}[d] & ... \\
... & & & &  & & & &  & & &  & ...
}
$$
The $\mathcal{W}$-algebra $ \mathcal{W}_{p,p'}$
is generated by all $\bullet$ (the socle part) and the vacuum vector $\times$. Clearly,
the socle part forms a nontrivial ideal in $\mathcal{W}(p,p')$.

\begin{conjecture} \label{conj-minimal}
Assume that $(p,p')=1$. The vertex algebra   $\WW_{p,p'}$ is $C_2$--cofinite with $ 2 p p' + \frac{(p-1) (p'-1)}{2}$--irreducible modules.
\end{conjecture}

\subsection{The triplet vertex algebra $\WW_{p,2}$ }

There are not many rigorous results about the  $\mathcal{W}$-algebras $\WW_{p,p'}$, except for $p'=2$ \cite{AM5}, \cite{AM8}.
We believe that some of the techniques introduced in \cite{AM5}, \cite{AM8}
are sufficient to prove the $C_2$-cofiniteness for all $p$ and $p'$.

The triplet vertex algebra $\WW_{p,2}$ can be realized as a subalgebra of $V_L$ generated by $\omega$ and primary vectors
$$ F = {\Q} e ^{-3 \a / 2}, \quad H = G F, \quad E = G ^2 F,$$
where $G$ is (new) screening operator defined by
$$ G=\sum_{i=1} ^{\infty} \frac{e ^{\a/2} _{-i} e ^ {\a/2}_i}{i} . $$

 Therefore, the triplet vertex algebra $\WW_{p,2}$ is $\mathcal{W}$--algebra of type
$$\mathcal{W}(2, h ,  h ,  h ), \quad (h = (2n+1) (pn+p-1)). $$

The next result shows that Conjecture \ref{conj-minimal} holds for $p'=2$.

\begin{theorem} (\cite{AM6}, \cite{AM8}) We have:
\begin{itemize}
\item[(1)] Every Virasoro minimal model for central charge $c_{p,2}$ is a module for $\WW_{p,2}$.
 \item[(2)] ${\WW}_{p,2}$ has exactly  $4 p + \frac{p-1}{2}$  irreducible modules.
\item[(3)] ${\WW}_{p,2}$ is  $C_2$--cofinite.
\item[(4)] ${\WW}_{p,2}$ is irrational and admits logarithmic modules.
\end{itemize}
\end{theorem}

Let $p=3$. Then $\WW_{3,2}$ is called the triplet $c=0$ vertex algebra.
Let us recall Zhu's algebra for this vertex algebra.

 Generators of $A(\WW_{3,2})$:
 $[\omega], [H], [E], [F].$

\begin{theorem} \cite{AM7}
Zhu's algebra $A(\WW_{3,2})$ decomposes as a direct sum:
$$A(\WW_{3,2}) = \bigoplus_{h \in  S ^{(2)} }   \mathbb{M}_{h}  \oplus \bigoplus _{h \in S ^{(1)}} \mathbb{I}_{h} \oplus {\C}_{-1/24},$$
$$ S ^{(2)} = \{ 5, 7, \tfrac{10}{3}, \tfrac{33}{8}, \tfrac{21}{8}, \tfrac{35}{24} \} ,
S ^{(1)} = \{  0, 1, 2, \tfrac{1}{3}, \tfrac{1}{8}, \tfrac{5}{8}, \tfrac{-1}{24} \}, $$
where $\mathbb{M}_{h}$ is ideal isomorphic to $M_2 (\C)$, $h \in S^{(2)}$,
 $\mathbb{I}_{h}$ is $2$--dimensional ideal, $h \in S^{(1)}$, $h \ne 0$, $h \ne -1/24$,
 ${\C}_{-1/24}$ is $1$-dimensional ideal,
$\mathbb{I}_{0}$ is $3$--dimensional ideal.
\end{theorem}

\begin{remark}
 The previous theorem shows that in the category of  $\WW_{3,2}$--modules, the projective cover of trivial representation should have $L(0)$--nilpotent rank three.  This result is used in the fusion rules analysis for the $c=0$ triplet algebra (cf. \cite{GRW-1}, \cite{GRW-2}).

 In \cite{AM8}, we proved that in the category of $\WW_{p,2}$--modules, the projective cover of every minimal model should have $L(0)$--nilpotent rank three. We expect  the same result to hold for general minimal $(p,p')$-models.
\end{remark}

\section{Construction of logarithmic modules and related problems}

\label{log-section}
In Section \ref{part-1} we propose a large family of (conjecturally) $C_2$-cofinite vertex algebras coming from integral lattices.
Now we examine indecomposable representations for these algebras.

\subsection{Progenerator and logarithmic modules}

A central question in representation theory of vertex algebra (or any algebraic structure) is to understand the structure
of indecomposable modules.   As it is well-known, for rational \footnote{Here for simplicity we assume strong rationality, meaning that for a given VOA every (weak) module is completely reducible.} vertex
algebras it is sufficient to classify irreducible modules. In contrast, for irrational $C_2$-cofinite  vertex algebras (with finitely many irreps)
it is essential to analyze the projective covers $P_i$ of irreducibles $M_i$, $ i \in {\rm Irr}$  \cite{Hu1}. Provided that we have a good  description of $P_i \rightarrow M_i$,  we can then form
a progenerator $P=\oplus_{i \in  {\rm Irr}} P_i$, and compute
$$\mathcal{A}={\rm End}_{V}(P)^{op},$$
which is known to be finite-dimensional.
This associative algebra plays a major rule in representation theory, and the least it gives the Morita equivalence of abelian categories
$${\rm f.g} \  V-Mod \  \cong  \ {\rm f.d.} \  \mathcal{A}-Mod.$$
As we shall see later, the same algebra is also important for purposes of modular invariance.
Because the category $V-Mod$ has a natural braided tensor category structure \cite{HLZ},
it is expected that one can do better and find a braided Hopf algebra $\mathcal{A}$
such that the above equivalence holds at the level of braided tensor categories
(this is known in some cases \cite{NT}, \cite{FGST2}, \cite{FGST3}).

The main problem here is that there is no good construction of $P_i$ even in the simplest case due to the fact that projective modules of  irrational $C_2$-cofinite vertex algebras are often {\em logarithmic}, that is, non-diagonalizable with respect to the Virasoro operator $L(0)$.
%The main problem here is that there is no good construction of $P_i$ even in the simplest case due to the fact that projective modules of  irrational $C_2$-cofinitel vertex algebras are {logarithmic}.
At the same time the $C_2$-cofinite vertex algebra is  conformally embedded inside a rational lattice vertex algebra, which is known to have no logarithmic modules.
Thus we cannot simply use the larger algebra to construct all relevant modules for the smaller algebra (except perhaps for the irreducibles \cite{AM1}).

Thus, in order to maneuver ourselves into a situation in which ${\rm End}_V(P)$ can be studied, we
first discuss construction of  general { logarithmic} modules.

\subsection{Screenings and logarithmic modules}

Here we propose a very general construction of logarithmic modules for vertex
algebras coming from screenings operators as in Section \ref{part-1}. As we shall see, in some cases these modules are indeed projective covers.
Our methods is based on screenings, local systems of vertex operators \cite{LL}, together with deformation of the vertex algebra action \cite{Li} (cf. also \cite{AM5}).
Conjecturally, the method introduced here is sufficient to construct all projective covers for
vertex algebras considered in Section \ref{part-1}.

Let $V$ be a vertex algebra of CFT type  and let $v \in V$  be a primary vector of conformal weight one, and
$$Y(v,x)=\sum_{n \in \mathbb{Z}} v_n x^{-n-1}.$$
As in \cite{Li} we let
\begin{equation} \label{delta-main}
\Delta(v,x) = x^{v_0} \exp \left( \sum_{n=1} ^{\infty}
\frac{v_n}{-n}(-x)^{-n} \right).
\end{equation}
If $v_0$ acts semisimply on $V$ and $w$ is its eigenvector, the expression $x^{v_0} w$ is defined as $x^{\lambda} w$,  where $\lambda$ is the corresponding eigenvalue.
But (\ref{delta-main}) is ambiguous if $v_0$ does not act semisimply,
Still the next result \cite{AM5} easily follows from \cite{Li}.

\begin{theorem} \label{delta}
Assume that $V$ and $v$ are as above. Let $\overline{V}$
be the vertex subalgebra of $V$ such that $\overline{V} \subseteq
\mbox{Ker}_V v_0$. Assume that $(M,Y_M)$ is
a $V$--module.
Define the pair $(\widetilde{M}, \widetilde{Y}_{\widetilde{M} })$
such that
$$\widetilde{M} = M \quad \mbox{as a vector space}, $$
$$ \widetilde{Y}_{\widetilde{M} } (a, x) = Y_{M} (\Delta(v,x) a, x) \quad \mbox{for} \ a \in \overline{V}. $$
 Then  $(\widetilde{M},
\widetilde{Y}_{\widetilde{M} } )$ is a  $\overline{V}$--module.
\end{theorem}

\begin{corollary}
Assume that $(M,Y_M)$ is a $V$--module such that $L(0)$ acts
semisimply on $M$. Then $(\widetilde{M},
\widetilde{Y}_{\widetilde{M} })$ is a logarithmic
$\overline{V}$--module if and only if $v_0$ does not act
semisimply on $M$.
\end{corollary}

By using this method  logarithmic $\mathcal{W}(p)_{Q}$-modules (including projective covers) can be constructed by taking $v=e^{-\alpha_i/\sqrt{p}}$ \cite{AM5}. But in general (cf. \cite{AM5}, \cite{NT})  one cannot construct all projective covers simply by taking $v$ to be a primary vector inside the generalized vertex algebra $V_{L^\circ}$. Instead we require
more complicated operators not present in the extended algebra $V_{L^\circ}$ (for a recent application of this circle of ideas see \cite{AM9}).
Then, when combined with $\mathcal{W}(p)_Q$ (and not all of $V_L$!), these more complicated local operators $v^{[i]}(z)$ (here $[i]$ has no particular meaning; it merely indicates some sort of "power" construction) became mutually local with $\mathcal{W}(p)_Q$, which allows us to extend our $\mathcal{W}$-algebra with $v^{[i]}(z)$ by using Li's theory of local systems \cite{LL}. Then we cook up a $\Delta$ operator and consider the residue
$$v^{[i]}_0=Res_{z_0} v^{[i]}(z),$$ which also annihilate $\mathcal{W}(p)_Q$, and again apply Theorem \ref{delta}. 

Already from this discussion we infer
\begin{corollary}
The vertex algebra $\mathcal{W}(p)_Q$ is irrational.
\end{corollary}
The previous result requires  a single screening $e^{-\alpha_i/\sqrt{p}}_0$.

\section{Some logarithmic modules for $\mathcal{W}(p)_Q$ }
In this section we shall describe a family of such logarithmic representations for $\mathcal{W}(p)_Q$  based on  the second power of screening operators.
 We present a new locality result which enables us to use concepts developed in \cite{AM5} and described above.
To exemplify the construction we only consider $Q=A_1$, and focus on $\mathcal{W}(p)$, but everything in this section applies to $\alpha$ replaced by $ \sqrt{p }\alpha_i$.
Here $p >1$. Define the following lattices
$$ L= {\Z} \a, \quad \widetilde{L} = {\Z} \tfrac{\a}{p}, $$
where $\langle \a, \a  \ra = 2 p. $

Then $V_{\widetilde{L} }$ has the structure of a generalized vertex operator algebra, and its subalgebra $V_L$ is a vertex operator algebra with the Virasoro vector
$$ \omega = \frac{1}{4p } \a (-1) ^2 + \frac{p-1}{2p} \a(-2).$$

 Let $a = e ^{-\a /p}$.
In the generalized vertex algebra $V_{ \widetilde{L}}$ the following locality relation holds:
$$ (z_1-z_2) ^{-2/p} Y(a, z_1) Y(a,z_2) - (z_2 - z_1) ^{-2/p}   Y(a,z_2) Y(a,z_1) = 0. $$

Define

\bea
&& \phi(t) = p t^{1/p}  \ {_2F_1} \left(\begin{array}{c} {1}/{p}, {2}/{p} \\ 1+ {1}/{p} \end{array}; t \right) =\sum_{j=0} ^{\infty} \frac{(-1) ^j }{j +1/p} t ^{j+ 1/p} { -2/p \choose j}, \nonumber \\
&& G(z) = \mbox{Res}_{z_1}\left( \phi(z/z_1) Y(a,z_1) Y(a,z) +  \phi(z_1/z)  Y(a,z) Y(a_1,z_1) \right) . \nonumber \eea
Let $G(z) = \sum_{n \in \Z} G(n) z ^{-n-1}$. Then, for  $n \in {\Z}$, we get

$$ G(n) =   \sum_{ j \ge 0} \frac{(-1) ^j }{ 1/p + j }{ -2/p \choose j} \left( a_{-1/p - j
+ n } a_{1/p +j} + a_{-1/p -j} a_{1/p  +j + n} \right), $$

$$ G(0) = 2 \sum_{ j \ge 0} \frac{(-1) ^j }{ 1/p + j }{ -2/p \choose j} a_{-1/p - j
 } a_{1/p +j}  . $$

We infer the following result:
 \begin{proposition} \label{tv-1} We have
\begin{itemize}
 \item[(1)]
$$[ L(n), G(m) ] = - m G(n+m). \qquad (m,n \in \Z).$$
In particular, $G(0)$ is a screening operator.

\item[(2)] The fields $G(z)$  and $L(z)$ are mutually local. More
precisely:

$$(z_1 -z_2) ^3 [L(z_1), G(z_2)] = 0. $$
We also have:

$$ L(z)_0 G(z) = G'(z), \quad  L(z)_1 G(z) = G(z), \quad L(z) _n G(z) = 0 \quad \mbox{for} \ n \ge 2. $$

\item[(3)] Let $\widetilde{L(n)} = L(n) + G(n)$. Then operators
$\widetilde{L(n)}$ define  on $$M^{+}_2=V_{L + \frac{p+1}{2p} \a} \oplus
V_{L + \frac{p-3}{2p} \a},  \quad  M^{-}_2=V_{L + \frac{1}{2p} \a} \oplus
V_{L + \frac{-3}{2p} \a} $$ the structure of the module for the
Virasoro algebra.

\end{itemize}
 \end{proposition}

\begin{remark}
 One can also represent $G(z)$ by using contour integrals as in \cite{NT}, and give a different proof of  Proposition \ref{tv-1} by using methods  developed in \cite{TK}. One defines
$$Q^{[2]}(z)=\int_{\gamma} dt Y(a,z)Y(a,tz)z $$
where
$\gamma$ is a certain contour. For $p \geq 2$, we can show that
$$ \frac{\Gamma(2/p+1)\Gamma(-1/p)}{\Gamma(1/p+1)} G(z)=Q^{[2]}(z),$$
where $\Gamma(z)$ is the usual $\Gamma$-function.
\end{remark}

\noindent Let $\widehat{a} = e ^{\a - \a/p}$. Define

\bea
 \overline{G}(z) &=& \sum_{n \in \Z} \overline{G}(n) z ^{-n-1} \nonumber \\ &=& \mbox{Res}_{z_1}\left( \phi(z/z_1) Y(a,z_1) Y(\widehat{a} ,z) +  \phi(z_1/z)  Y(\widehat{a},z) Y(a,z_1) \right) . \nonumber  \eea

\noindent Then
$$ \overline{G}(n)=   \sum_{ j \ge 0} \frac{(-1) ^j }{ 1/p + j }{ -2/p \choose j} \left( \widehat{a}_{-1/p - j
+n  } a_{1/p +j} + a_{-1/p -j} \widehat{a}_{1/p  +j +n} \right). $$
Let
$$  \mu = \frac{p}{p-1}. $$

First we need the following result.

\begin{lemma} \label{com-2}  We have the following relations:
\begin{itemize}
\item[(i)]  $[ e ^{\a}_0 , G(n)] =- n \mu   \overline{G}(n-1)$; i.e.,  $ [e ^{\a}_0, G(z)] = \mu \overline{G} ' (z)$;

\item[(ii)]   $ [ e^{\a}_0, G(0)] = 0$,
\item[(iii)] $[ L(n), \overline{G}(m) ] = -(n+m+1) \overline{G}(n+m) $.
\end{itemize}
\end{lemma}
\begin{proof}
Let us prove relation (i). First we notice that
$$ e ^{\a}_0 a = \mu D \widehat{a}. $$
Then we have

\bea
[e ^{\a}_0, G(z)] &= & \mu  \mbox{Res}_{z_1}  ( \phi(z/z_1)  \partial_{z_1}  Y(\widehat{a},z_1)  Y(a,z)
  +  \phi(z_1/z)  Y(a,z)   \partial_{z_1}  Y(\widehat{a},z_1)   )  \nonumber \\   &&+\mu \mbox{Res}_{z_1}  ( \phi(z/z_1)  Y(a,z_1)  \partial_z  Y(\widehat{a},z)
 + \phi(z_1/z)  \partial_z  Y(\widehat{a},z)    Y(a,z_1) )   \nonumber \\
 & = & -\mu \mbox{Res}_{z_1}   z ^{1/p} z_1 ^ {1/p -1} \nonumber \\  && \left( (z_1-z) ^{-2/p}  Y(\widehat{a},z_1) Y(a,z) - (z-z_1) ^{-2/p} Y(a,z) Y(\widehat{a},z_1) \right)  \nonumber \\
 & & + \mu \mbox{Res}_{z}  z ^{1/p-1} z_1 ^ {1/p} \nonumber \\ & & \left( (z_1-z) ^{-2/p} Y(a,z_1) Y(\widehat{a},z) - (z-z_1) ^{-2/p}   Y(\widehat{a},z) Y(a,z_1) \right) \nonumber \\
 & & + \mu \mbox{Res}_z ( \overline{G}(z) ) \nonumber \\
 & = & \mu \mbox{Res}_z ( \overline{G}(z) ) . \nonumber
\eea
This proves relation (i). The relation (ii) follows from (i). The proof of (iii) is similar to that of (ii).
\end{proof}

Recall that the triplet vertex algebra $\mathcal{W}(p)$ is realized as a subalgebra of $V_L$ generated by the vectors
$$ \omega, \ F = e ^{-\a}, \ H = {\Q} F, \ E = {\Q} ^2 F, $$
where $Q =e ^{\a}_0$.

The doublet vertex algebra $\mathcal{A}(p)$ is the subalgebra of $V_{ \widetilde{L} }$ generated by
$$ x ^- = e ^{-\a/2}, \quad x^+ = {\Q} a ^{-\a /2}. $$

Clearly, $\mathcal{W} (p)$ is a subalgebra of $\mathcal{A}(p)$.
\begin{proposition}We have
\begin{itemize}
\item[(i)] The fields
$ G(z), \overline{G}(z),  Y(x^-,z), Y(x ^+ ,z)$ are mutually local.
\item[(ii)]
\bea
U &=& \mbox{span}_{\C} \{ G(z);   Y(v,z) \ \vert \ v \in \mathcal{W}(p)  \}, \nonumber \\
U^{e} &=& \mbox{span}_{\C} \{ G(z), \overline{G}(z);   Y(v,z) \ \vert \ v \in \mathcal{W}(p)  \}  \nonumber
\eea
are local subspaces of fields acting on $M_2 ^{\pm}$.
\end{itemize}
\end{proposition}
\begin{proof}
It is only non-trivial to prove that $G(z)$ and $Y(x ^+,z)$ are local.
We have
 \bea
&& [ Y(x^+,z_1), G(z_2) ]  \nonumber \\
 && =[{\Q}, [Y(x^-,z_1), G(z_2)]] - [Y(x ^-,z_1), [{\Q},G(z_2)] ] = - [Y(x ^-,z_1), [{\Q},G(z_2)] ].
 \eea
 The proof easily follows if we invoke Lemma \ref{com-2} and  the fact that the fields
 $Y(a ^-,z_1)$ and $ e^{\a - \a/p}(z)$ are local.  This proves (i).
 By using a standard result on locality of vertex operators \cite{Li-local}, \cite{LL}, we invoke that the field $G(z)$, $\overline{G}(z)$  are local with all fields $Y(v, z)$, $v \in \mathcal{A}(p)$. In particular, the sets $U$  and $U ^e$ are local.
 \end{proof}

 \begin{remark}
 We believe that this locality result is new. One can see that $G(z)$ is local only with $\triplet$, but it is not local with all fields $Y(a,z)$, $a \in V_L$. In particular, $G(z)$ is not local with Heisenberg field $\alpha(z)$.
 \end{remark}

 Let $\mathcal{V}$  (resp. $\mathcal{V} ^e$) be the vertex algebra generated by local subspace $U$ (resp. $U ^e$). It is clear that
$$v \mapsto Y(v,z) \ (v \in \mathcal{W}(p) )$$
is a injective homomorphism of vertex algebras. So $\mathcal{W}(p)$ can be considered as a subalgebra of $\mathcal{V}$.

\begin{theorem} We have:
\begin{itemize}
\item[(i)] $ \mathcal{W} (p). G(z) = \mbox{span}_{\C} \{ Y(v,z) _n G(z) \  \ \vert \ v \in  \mathcal{W} (p), \ n \in \Z \} \cong \Pi(p-1). $

\item[(ii)] $\mathcal{V} \cong \mathcal{W}(p) \oplus \Pi(p-1)$.

\item[(iii)] $\mathcal{V}^e = \mathcal{W}(p) \oplus E$, where  $E = \mathcal{W}(p). \overline{G}(z). $

\item[(iv)]  There is a non-split extension

$$ 0 \rightarrow \Pi(p-1) \rightarrow E \rightarrow \Lambda(1) \rightarrow 0. $$
\end{itemize}
\end{theorem}
\begin{proof}
It is clear that $\mathcal{V} = \mathcal{W}(p) \oplus  \mathcal{W} (p). G(z)$. So it remains to identify cyclic $\mathcal{W}(p)$--module $\mathcal{W}(p). G(z)$.
The locality relations
$$  (z_1 -z_2) ^{2p-1} [Y(x,z_1), G(z_2)] = 0, \quad (x \in \{E,F,H \} ),$$
imply that $\mathcal{W} (p). G(z)$ is a $\N$--graded $\mathcal{W}(p)$--module with lowest weight $1$. Top component is $2$--dimensional and spanned by $G(z)$ and $[{\Q},G(z)]$. By using representation-theoretic results from \cite{AM1} we see that this module is isomorphic to $\Pi(p-1)$, and that $E / \Pi(p-1) \cong \Lambda(1)$.
 The proof follows.
\end{proof}

\begin{remark}
We know that there is also a non-split extension
$$ 0 \rightarrow \Pi(p-1) \rightarrow V_{L + \a - \a/p} \rightarrow \Lambda(1) \rightarrow 0. $$
But, $V_{L + \a - \a/p} \ncong E$.
\end{remark}

The operators $\widetilde{G(z)}_n$, $n \in \Z$, define on $\mathcal{V}$ the structure of a module for the Heisenberg algebra such that $\widetilde{G(z)}_0$ acts trivially.
Therefore the field
$$ \Delta(\widetilde{G(z)},z_1) = z_1 ^ {\widetilde{G(z)}_0} \exp  \left(\sum_{n = 1} \frac{\widetilde{G(z)} _{n} }{-n} z_1 ^{-n}  \right) $$ is well defined on $\mathcal{V}$. As in \cite{AM5} we have the following result:

\begin{theorem}
Assume that $(M,Y_M(\cdot,z_1))$ is a weak ${\mathcal V}$--module. Define the pair $(\widetilde{M}, \widetilde{Y}_{\widetilde{M}}(\cdot, z_1) )$ such that
\bea &&  \widetilde{M}= M \qquad \mbox{ as a vector space}, \nonumber \\
&& \widetilde{Y}_{\widetilde M} (v(z),z_1) = Y_{ M} (\Delta(\widetilde{G(z)},z_1) v(z),z_1). \nonumber \eea

Then $(\widetilde {M}, \widetilde{Y}_{\widetilde{M}}(\cdot, z_1) )$
is a weak ${\mathcal V}$--module. In particular, $(\widetilde {M}, \widetilde{Y}_{\widetilde{M}}(\cdot, z_1) )$ is a ${\mathcal W}(p)$--module.
\end{theorem}

Recall that $M_2 ^{\pm}$ are modules for the vertex algebra ${\mathcal V}$ with the vertex operator map
$$ Y (v(z), z_0 ) = v(z_0), \qquad v(z) \in {\mathcal V}.$$
Applying the above construction we get a (new) explicit  realization of logarithmic modules for $\mathcal{W}(p)$.
 \begin{theorem}
 $ (\widetilde{M_2 ^{\pm} }, \widetilde{Y})$ is a $\mathcal{W}(p)$--module  such that
 $$\widetilde{Y}(v,z) =  Y(v,z) + \sum_{n=1} ^{\infty} \frac{G(z)_n Y(v,z)}{-n} (-z) ^{-n}, \quad v \in \mathcal{W}(p). $$
 In particular,
 $$\widetilde{Y} (\omega, z) = \widetilde{L}(z). $$
 The operator $\widetilde{L(0)}$ acts on $\widetilde{M_2 ^{\pm}}$ as  $$\widetilde{L(0)} = L(0) + G(0)$$
 and it has nilpotent rank $2$.
 \end{theorem}

\begin{remark}By applying the methods developed in  \cite{AM5} we see that
  $\widetilde{M_2 ^{\pm}}$ are self-dual, logarithmic modules of semisimple rank three. Moreover,  $\widetilde{M_2 ^{+}}$ (resp. $\widetilde{M_2 ^{-}}$ ) is projective cover of $\Lambda(2)$ (resp. $\Pi(p-2)$). The same modules have been constructed in \cite{NT} by using a slightly different method.
  \end{remark}

\section{Conclusion}

We hope that we have conveyed the main ingredients  behind the plethora of $\mathcal{W}$-algebras connected to Logarithmic
Conformal Field Theory. There are still numerous problems to be resolved at the structural
level (e.g. $C_2$-cofiniteness), but we hope that the present techniques in vertex algebra theory - with further constructions as in Chapter 6 - are sufficient to
resolve the main conjectures in the paper, including construction of projective covers. Eventually this development on the vertex algebra
side will play an important role in finding a precise relationship with the finite-dimensional quantum groups at root of unity proposed in  \cite{FGST1}-\cite{FGST3}, \cite{FT}.

There are several aspects of $C_2$-cofinite $W$-algebras that we did not discuss in this paper. Here we briefly outline on these developments.

\begin{itemize}

\item[(1)] There is an important (simple-current) extension of the triplet vertex algebra $\mathcal{W}(p)$, called  the doublet $\mathcal{A}(p)$. If $p$ is even, that $\mathcal{A}(p)$ carries the structure of a vertex algebra (or vertex superalgebra). Its representation
theory has been developed in \cite{AM10}. This extension can be constructed in the higher rank as well. Also, a large portion of the present work extends to $N=1$ vertex operator superalgebras.

\item[(2)]  We expect to see rich combinatorics underlying $\mathcal{W}(p)_Q$, including properties of graded dimensions of modules and of some distinguished subspaces examined in \cite{MP2}. Another important facet of the theory was initiated  in \cite{AM6}, \cite{AM8} in connection to constant term identities of Morris-Macdonald type (see also \cite{CLWZ}). These identities are expected to play a role in the theory of higher Zhu's algebras.

\item[(3)]  Modular invariance and one-point functions on the torus are important ingredients in CFT \cite{Zh}.
In \cite{AM4} (cf. also \cite{F}) we have shown that the space of one-point functions for $\mathcal{W}(p)$ is $3p-1$ dimensional.
But in view of \cite{Miy1}, it is not completely obvious how to describe the space of one-point functions explicitly
via certain pseudotraces. For $\mathcal{W}_{p,p'}$ we still do not know precisely even its dimension, although there is
an obvious guess by looking at the properties of irreducible characters \cite{FGST1}-\cite{FGST3}. One-point functions for
the $C_2$-cofinite vertex algebra $SF^+$  coming from {\em symplectic fermions} \cite{Abe} have been recently studied in \cite{AN}.
Some general results about "logarithmic modular forms" are obtained in \cite{KM}.

\item[(4)] There is ongoing effort in the direction of constructing the {\em full} rational conformal field theory \cite{HK1}-\cite{HK2}.
Although it is not clear how to generalize the notion of full field algebra to general $C_2$-cofinite vertex algebras, some progress has been
achieved recently on the construction of the bulk space, in the case of the triplet vertex algebra $\mathcal{W}(p)$ and $\mathcal{W}_{2,3}$ \cite{GR}, \cite{GRW-1}, \cite{GRW-2}, (cf. also \cite{W} for general $p$ and $p'$). 

\end{itemize}


\begin{thebibliography}{FGST3}


\bibitem[Abe]{Abe} T. Abe, A $Z_2$-orbifold model of the symplectic fermionic vertex operator superalgebra,
{\em Mathematische Zeitschrift}, {\bf  255} (2007).

\bibitem[ABD]{ABD} T. Abe, G. Buhl, C. Dong, Rationality, regularity and $C_2$ -cofiniteness, {\em Trans. Amer. Math. Soc.} {\bf 356} (2004) 3391-3402.

\bibitem[Ad]{Ad}  D. Adamovi\'{c}, Classification of irreducible modules of certain subalgebras of
free boson vertex algebra, J. Algebra 270 (2003) 115-132.

\bibitem[An]{A}
G. Andrews,  {\em q-Series: Their Development and Application in Analysis, Number Theory, Combinatorics,
Physics, and Computer Algebra}, CBMS Regional Conf. Ser. in Math., vol. 66, American Mathe-
matical Society, Providence, RI, 1986.

\bibitem[Ar]{Ar} T. Arakawa, Representation Theory of $W$-Algebras,  {\em Inventiones Math.},  {\bf 169} (2007), 219-320 {\tt arXiv:math/0506056}.

\bibitem[AM1]{AM1} D. Adamovic and A. Milas , On the triplet vertex algebra $W(p)$, {\em Advances in Mathematics} {\bf 217} (2008), 2664-2699.

\bibitem[AM2]{AM2}  D. Adamovic and A. Milas, The $N=1$ triplet vertex operator
superalgebras,  {\em Communications in Mathematical
Physics}, {\bf 288} (2009), 225-270.

\bibitem[AM3]{AM3} D. Adamovic and A. Milas, The $N=1$ triplet vertex operator
superalgebras: twisted sector, {\em Symmetry, Integrability and
Geometry: Methods and Applications} (SIGMA) (Special Issue on
Kac-Moody Algebras and Applications), {\bf 4} (2008), 087, 24 pages.

\bibitem[AM4]{AM4} D. Adamovic and A. Milas, An analogue of modular BPZ
equation in logarithmic conformal field theory,  {\em
Contemporary Mathematics}, {\bf 497} (2009), 1-17.

\bibitem[AM5]{AM5} D. Adamovic and A. Milas, Lattice construction of logarithmic
modules for certain vertex algebras,  {\em Selecta
Mathematica (N.S.)}, 15 (2009), no. 4, 535-561.

\bibitem[AM6]{AM6} D. Adamovic and A. Milas, On $W$-algebras associated to $(2,p)$ minimal models and, International Mathematics Research Notices (2010) 20, 3896-3934.

\bibitem[AM7]{AM7}   D. Adamovic and A.Milas, The structure of Zhu's algebras for certain W-algebras,
Advances in Mathematics,  {\bf 227} (2011)  2425-2456.

\bibitem[AM8]{AM8}  D. Adamovic and A.Milas, On $W$-algebra extensions of $(2,p)$ minimal models: $p > 3$ ,
{\em Journal of Algebra}  {\bf 344} (2011) 313-332.

\bibitem[AM9]{AM9}  D. Adamovic and A.Milas, An explicit realization of logarithmic modules for the vertex operator algebra $\mathcal{W}_{p,p'}$
, J. Math. Phys. 53, 073511 (2012).

\bibitem[AM10]{AM10}   D. Adamovic and A.Milas, The doublet vertex operator algebra $\mathcal{A}(p)$ and $\mathcal{A}_{2,p}$, submitted to {\em Contemporary 
Mathematics}.

\bibitem[AN]{AN} Y. Arike and A. Nagatomo,  Some remarks on pseudo-trace functions for orbifold models associated with symplectic fermions, preprint  {\tt arXiv:1104.0068}.

\bibitem[AMOS]{AMOS} H. Awata, Y. Matsuo, S. Odake, J. Shiraishi, Excited States of Calogero-Sutherland Model and Singular Vectors of the $W_N$ Algebra,
Nucl.Phys. B449 (1995) 347-374, {\tt arxiv.org/9503043}.

%\bibitem[Bar]{Bar} K. Barron,  A supergeometric interpretation of vertex operator superalgebras, {\em IMRN},  Issue {\bf 9}, (1996) 409-430.

\bibitem[BT]{BT} J. de Boer and T. Tjin, Quantization and representation theory of finite W-algebras,
Comm. Math. Phys. 158 (1993), 485–516

\bibitem[Bo]{Bo}
R. E. Borcherds,  Vertex algebras, Kac--Moody algebras, and the
Monster, {\em Proc. Nat. Acad. Sci. U.S.A.} {\bf 83} (1986), 3068--3071.

\bibitem[CF]{CF} 
N. Carqueville and M. Flohr, Nonmeromorphic operator product expansion and $C_2$-cofiniteness for a family of W-algebras, {\em Journal of Physics A: Mathematical and General} {\bf 39} (2006): 951.

\bibitem[CLWZ]{CLWZ}
T. Chappell, A.  Lascoux, S. Warnaar, W. Zudilin,  Logarithmic and complex constant term identities, {\tt arXiv:1112.3130}.


\bibitem[CR]{CR} T. Creutzig and D. Ridout, Relating the Archetypes of Logarithmic Conformal Field Theory, {\tt arXiv:1107.2135}.

\bibitem[DeK]{DeK} A. De Sole and V. Kac, Finite vs. affine
$W$-algebras, {\em Japanese Journal of Math.} {\bf 1} (2006)
137-261; {\tt arxiv:math/0511055.}


\bibitem[D]{D} C. Dong, Vertex algebras associated with even lattices,  {\em J. Algebra} {\bf 160} (1993), 245-265.

\bibitem [DL]{DL} C. Dong and J. Lepowsky,
{\em Generalized vertex algebras and relative vertex operators}, Progress in Mathematics,
Birkh\"auser,  Boston, 1993.

\bibitem[DLM1]{DLM1} C. Dong, H. Li and G. Mason, Vertex operator algebras and
associative algebras, {\em J. Algebra} {\bf 206} (1998), 67-96.

\bibitem[DLM2]{DLM2} C. Dong, H. Li and G. Mason, Modular invariance of trace functions in orbifold theory
and generalized Moonshine, Comm. Math. Phys. 214 (2000), 1-56.


\bibitem[FeS]{FSE} B. Feigin and A. Stoyanovsky, Quasi-particles
models for the representations of Lie algebras and geometry of flag
manifold; arXiv:hep-th/9308079.

\bibitem[FFL]{FFL} B. Feigin, E. Feigin, and P. Littelmann, Zhu's algebras, $C_2$-algebras and abelian radicals, 
Journal of Algebra 329.1 (2011): 130-146, {\tt arXiv:0907.3962}.

\bibitem[FL]{FL} E. Feigin, and P. Littelmann, Zhu's algebra and the $C_2$-algebra in the symplectic and the orthogonal cases,  J. Phys. A: Math. Theor. 43 (2010) 135206 {\tt arXiv:0911.2957v1}.


\bibitem [FGST1]{FGST1} B.L. Feigin, A.M. Ga\u\i nutdinov, A. M. Semikhatov, and I. Yu Tipunin, I,
The Kazhdan-Lusztig correspondence for the representation category
of the triplet $W$-algebra in logorithmic conformal field
theories. (Russian) Teoret. Mat. Fiz. {\bf 148} (2006), no. 3,
398--427.

\bibitem [FGST2]{FGST2} B.L. Feigin, A.M. Ga\u\i nutdinov, A. M. Semikhatov, and I. Yu Tipunin,
Logarithmic extensions of minimal models: characters and modular
transformations. {\em Nuclear Phys.} B {\bf 757} (2006), 303--343.

\bibitem [FGST3]{FGST3} B.L. Feigin, A.M. Ga\u\i nutdinov, A. M. Semikhatov, and I. Yu Tipunin,
Modular group representations and fusion in logarithmic conformal
field theories and in the quantum group center. {\em Comm. Math.
Phys.} {\bf 265} (2006), 47--93.

\bibitem[FT]{FT} B. Feigin and I. Tipunin,  Logarithmic CFTs connected with simple Lie algebras,
{\tt arXiv:1002.5047}.


\bibitem[FFHST]{FFHST} J. Fjelstad, J. Fuchs, S. Hwang, A.M.
Semikhatov and I. Yu. Tipunin, Logarithmic conformal field theories
via logarithmic deformations, {\em  Nuclear Phys.} B {\bf 633} (2002),
379--413.


\bibitem[F]{F} M. Flohr, On modular invariant partition functions of conformal field theories with logarithmic operators. International Journal of Modern Physics A 11.22 (1996): 4147-4172.

\bibitem[FGK]{FGK}  M. Flohr, A. Grabow and R. Koehn, Fermionic Expressions for the Characters of $c(p,1)$ Logarithmic Conformal Field Theories,
{\em Nucl.Phys.}  B{\bf 768}  (2007) 263-276.


\bibitem[FrB]{FrB} E. Frenkel and D. Ben-Zvi, {\em Vertex algebras and algebraic curves},
Mathematical Surveys and Monographs, 88, American Mathematical
Society, Providence, RI, 2001.


\bibitem[FKW]{FKW} E. Frenkel, V. Kac and M. Wakimoto, Characters and fusion rules for W-algebras via
quantized Drinfeld-Sokolov reduction, Commun. Math. Phys. 147 (1992), 295--328.

\bibitem[FKRW]{FKRW} E. Frenkel, V. Kac, A. Radul and W.Wang, $W_{1+\infty}$ and $W(gl_N)$ with central charge $N$, {\em Comm.
Math. Phys.} {\bf 170} (1995), 337-357.


\bibitem[FHL]{FHL}
I.~B. Frenkel, Y.-Z. Huang and J.~Lepowsky,
On axiomatic approaches to vertex operator algebras and modules,
{\em Memoirs Amer. Math. Soc.} {\bf 104}, 1993.

\bibitem[FLM]{FLM}
I.~B. Frenkel, J.~Lepowsky, and A. Meurman,
{\em Vertex Operator Algebras and the Monster},
Pure and Appl. Math., Vol. 134, Academic Press, New York, 1988.

\bibitem[Fu]{Fu} J. Fuchs,  On non-semisimple fusion rules and tensor categories, Contemporary Mathematics
442  (2007)  315--337.

\bibitem[FS]{FS} J. Fuchs and C. Schweigert,  Hopf algebras and finite tensor categories in conformal field theory,
Revista de la Uni—n Matem‡tica Argentina 51  (2010)  43--90.

\bibitem[FHST]{FHST} J. Fuchs, S. Hwang, A.M. Semikhatov and I. Yu. Tipunin,
Nonsemisimple Fusion Algebras and the Verlinde Formula, { Comm.
Math. Phys.} {\bf 247} (2004), no. 3, 713--742.


\bibitem[GR]{GR} M. Gaberdiel and I. Runkel, From boundary to bulk in logarithmic CFT, J. Phys. A41 (2008) 075402.

\bibitem[GRW1]{GRW-1} M. Gaberdiel, I. Runkel and S. Wood, Fusion rules and boundary conditions in the $c=0$ triplet model, J. Phys. A: Math. Theor. 42 (2009) 325--403,
{\tt arxiv}: 0905.0916.

\bibitem[GRW2]{GRW-2} M. Gaberdiel, I. Runkel and S. Wood, A modular invariant bulk theory for the $c=0$ triplet model, J.Phys. A: Math. Theor. 44 (2011) 015204
, {\tt arxiv}: 1008.0082v1


\bibitem [Ga]{Gn} T. Gannon,  {\em Moonshine beyond the Monster: The Bridge Connecting Algebra, Modular Forms and Physics}, Cambridge University Press, 2006.


\bibitem[GG]{GG} M. Gaberdiel and T. Gannon, Zhu's algebra, the $C_2$-algebra, and twisted modules, {\em Contemporary Mathematics}, {\bf 497}, 2009.

%\bibitem[Gu]{Gu} V. Gurarie, Logarithmic operators in conformal field theory,
%{\em Nucl.Phys.} {\bf B410} (1993), 535-549.


\bibitem[Hu1]{Hu1} Y.-Z. Huang,  Cofiniteness conditions, projective covers and the logarithmic tensor product theory, J. Pure Appl. Algebra {\bf 213} (2009) 458-475;
 {\tt arxiv.0712.4109}.

\bibitem[Hu2]{Hu2} Y.-Z. Huang, Generalized twisted modules associated to general automorphisms of a vertex operator algebra, preprint, {\tt arXiv:0905.0514}.

%\bibitem[48]{Hu3}
%Y.--Z. Huang, {\em Two-dimensional conformal geometry and vertex operator algebras.} Progress in %Mathematics, 148. Birkh�user Boston, Inc., Boston, MA, 1997.

\bibitem[HK1]{HK1} Y.-Z. Huang and L. Kong,  Full field algebras,  Comm. Math. Phys. 272 (2007)
345-396.

\bibitem[HK2]{HK2} Y.-Z. Huang and L. Kong,  Modular invariance for conformal full field algebras,  {\em T.A.M.S.},
Volume 362, (2010), 3027�3067.

\bibitem[HLZ]{HLZ}  Y.-Z. Huang, J. Lepowsky and L. Zhang,  Logarithmic tensor product theory for generalized modules for a conformal vertex algebra, {\tt arXiv:0710.2687}.
(also Parts I-VIII:  arXiv:1012.4193, arXiv:1012.4196, arXiv:1012.4197,  arXiv:1012.4198,  arXiv:1012.4199, arXiv:1012.4202,  arXiv:1110.1929,  arXiv:1110.1931).


\bibitem[HLLZ]{HLLZ}   Y.-Z. Huang, H. Li,  J. Lepowsky and L. Zhang,
On the concepts of intertwining operator and tensor product module in vertex operator algebra theory,
{\em Journal of Pure and Applied Algebra}, Volume 204, (2006), 507-535.


\bibitem[Kac]{Kac}  V. Kac, {\em Vertex algebras for beginners}, 2nd edition, {AMS}, 1998.

 \bibitem[KW]{KW}
V. Kac and W. Wang, Superconformal vertex operator superalgebras
and their representations, {\em Contemp. Math.} {\bf 175}, 1994.


\bibitem[KL]{KL} D. Kazhdan and G. Lusztig, Affine Lie algebras and quantum groups, {\em International Mathematics Research Notices}, 1991:21-29.

\bibitem[KM]{KM} M. Knopp and G. Mason,  Logarithmic vector-valued modular forms , {\tt  arXiv:0910.3976}.

%\bibitem[47]{KS} H. Kondo and A. Saito, Indecomposable decomposition of tensor products of modules over %the restricted quantum universal enveloping algebra associated to ${\mathfrak{sl}_2}$,  {\tt arXiv:0901.4221}.


%\bibitem[48]{La} A. Lachowska, A counterpart of the Verlinde algebra for the small quantum group. %{\em Duke Math. J.}  {\bf 118 }
%(2003), no. 1, 37�60.

\bibitem[LL]{LL} J. Lepowsky and H. Li, {\em Introduction to Vertex Operator Algebras and Their Representations},
Birkh\"auser, Boston, 2003.

\bibitem[Li1]{Li-local} H. Li, Local systems of vertex operators, vertex superalgebras and modules, {\em J. Pure Appl. Algebra} {\bf 109} (1996), 143--195

\bibitem[Li2]{Li} H. Li, The physics superselection principle in vertex operator algebra theory,  {\em J. Algebra} {\bf 196} (1997), 436--457.

\bibitem[Li3]{Li2} H. Li, Representation theory and tensor product theory for
vertex operator algebras, PhD thesis, Rutgers University, 1994.

\bibitem[M1]{M1} A. Milas, Weak modules and logarithmic intertwining operators for vertex operator algebras. Recent developments in infinite-dimensional Lie algebras and conformal field theory (Charlottesville, VA, 2000),
201--225, { Contemp. Math.} {\bf 297}, Amer. Math. Soc., Providence,
RI, 2002.

\bibitem[M2]{M2}
{ A. Milas}, Fusion rings
associated to degenerate minimal models, {\em Journal of Algebra} {\bf 254} (2002), 300-335

\bibitem[M3]{M3} A. Milas, Logarithmic intertwining operators and
vertex operators,  {\em Comm. Math. Physics}  {\bf 277}, (2008)
497-529.

%\bibitem[P]{MP1}  M.Penn, Lattice vertex algebras and combinatorial bases, submitted.

\bibitem[MP]{MP2} A.Milas and M. Penn, Lattice vertex algebras and combinatorial bases: general case and $W$-algebras, {\em New York Journal of Mathematics} {\bf 18}, (2012) 621-650.


\bibitem[Miy1]{Miy1} M. Miyamoto,  {\em Modular invariance of vertex operator algebras satisfying $C\sb
2$-cofiniteness}. {Duke Math. J.} {\bf 122} (2004), 51--91.


\bibitem[Miy2]{Miy2} M.Miyamoto, Flatness of Tensor Products and Semi-Rigidity for $C_2$-cofinite Vertex Operator Algebras I, {\tt  arXiv:0906.1407}.


\bibitem[NT]{NT} K. Nagatomo and A. Tsuchiya, The Triplet Vertex Operator Algebra $W(p)$ and the Restricted Quantum Group at Root of Unity,
Exploring new structures and natural constructions in mathematical physics, 1–49, Adv. Stud. Pure Math., 61,
Math. Soc. Japan, Tokyo, 2011., {\tt arXiv:0902.4607}.

%\bibitem[MP]{MP} M. Penn, Lattice vertex algebras and combinatorial bases, Ph.D thesis, UAlbany 2011 (advisor: A.Milas).


%\bibitem[Se]{Se} G. Segal, Definition of conformal field theory, preprint 1988.

\bibitem[Sem]{Sem} A. Semikhatov,   Virasoro central charges for Nichols algebras, {\tt arXiv:1109.1767 }.


\bibitem[TK]{TK} Tsuchiya, A. and Kanie, Y., {Fock space representations of
the Virasoro algebra - Intertwining operators},  {\em Publ. RIMS} {\bf
22} (1986), 259-327.

\bibitem[W1]{W1} W. Wang, Classification of irreducible modules of $W_3$ algebra with $c = -2$, {\em Comm. Math. Physics} {\bf 195} (1998), 113--128.

\bibitem[W2]{W2} W. Wang, Nilpotent orbits and finite W-algebras, Fields Institute Communications Series 59 (2011), 71--105

\bibitem[W]{W} S. Wood, Fusion rules of the $W_{p,q}$ triplet models, {\em Journal of Physics A} 43.4 (2010): 045212.

\bibitem[Zh]{Zh}
Y. Zhu, Modular invariance of characters of vertex operator algebras,
{\em J. Amer. Math. Soc.} {\bf 9} (1996), 237-302.


\end{thebibliography}
\end{document}